\documentclass[a4paper,10pt]{amsart}

\usepackage{
amssymb, amsfonts, amscd, 
mathrsfs, color, graphicx}

\usepackage[symbol]{footmisc}

\newtheorem{theo}{Theorem}
\newtheorem{lem}[theo]{Lemma}
\newtheorem{prop}[theo]{Proposition}

\newtheorem{rem}[theo]{Remark}

\theoremstyle{definition}

\newtheorem{example}{Example}

\newcommand{\R}{\mathbb{R}}

\newcommand{\Z}{\mathbb{Z}}
\newcommand{\N}{\mathbb{N}}

\newcommand{\Co}{\mathcal{C}}
\newcommand{\G}{\mathcal{G}}

\newcommand{\T}{\mathcal{T}}           
\newcommand{\J}{\mathcal{J}}           
\newcommand{\pj}{\mathcal{P}}          
\newcommand{\B}{\mathcal{B}}   
\newcommand{\M}{\mathcal{M}}

\newcommand{\ZZD}{\mathbb{Z}^2}

\newcommand{\E}{\mathcal{E}}
\newcommand{\F}{\mathcal{F}}
\newcommand{\Hil}{\mathcal{H}}

\newcommand{\Proj}{\mathbb{P}}

\newcommand{\vsp}{\textnormal{span}}
\newcommand{\iso}{\mathcal{T}}

\newcommand{\tr}[1]{T_{#1}}
\newcommand{\rot}[1]{R_{#1}}
\newcommand{\RR}{\R^d}         

\newcommand{\esssup}{\textnormal{ess}\!\!\sup}

\newcommand{\ol}[1]{\overline{#1}}

\newcommand{\wh}[1]{\widehat{#1}}

\title[Learning smooth subspaces]{Learning optimal smooth invariant subspaces for data approximation}
\author{D. Barbieri, C. Cabrelli, E. Hern\'andez, U. Molter}

\date{\today}

\subjclass[2020]{41A65, 43A70} 

\keywords{Invariant subspaces; data approximation; Paley-Wiener spaces; optimal subspaces}

\begin{document}

\maketitle

\allowdisplaybreaks[2]

\begin{abstract}
	 In this article, we consider the problem of approximating a finite set of data (usually huge in applications) by invariant subspaces generated through a small set of smooth functions. The invariance is  either by translations under a full-rank lattice or through the action of crystallographic groups. Smoothness is ensured by stipulating that the generators belong to a Paley-Wiener space, that is selected  in an optimal way based on the characteristics of the given data.
To complete our investigation, we analyze the fundamental role played by the lattice in the process of approximation.

\end{abstract}

\section{Introduction}

In various signal and image processing applications, we frequently encounter a specific class of signals generated by the measurement of physical phenomena, such as electrocardiograms, audio signals, radar data, medical images, and internet traffic, to name a few. These signals within each class stem from a common source, leading us to expect that they can be effectively characterized by a concise set of parameters.

Furthermore, certain classes of signals exhibit inherent invariance properties, such as translation or rotation invariance, which are particularly prominent in the case of natural images. Consequently, when seeking a simple model or subspace that closely approximates a specific class of signals, it is natural to demand that this subspace also embodies these inherent invariances.

The endeavor to learn an optimal subspace with specific properties from a dataset has been explored in various contexts. This quest was initially addressed in \cite{ACHM2007}, where it was assumed that the class of signals belonged to the space $L^2(\R^d).$ 
In this early study, an optimal 'small' subspace was constructed, which was required to be invariant under translations along a lattice. The measure of "closeness" was determined by a functional, as we will elaborate on shortly.

Subsequently, in \cite{Tessera, CM2016}, the scope was broadened to encompass cases where the approximation subspaces also possessed extra invariance, specifically involving translations on a finer lattice. This approximation problem was also considered in general Banach spaces in \cite{Cuenya}.
More recently, as demonstrated in \cite{BCHM2020} a more general scenario was considered, encompassing invariance properties like translations and rotations, such as those associated with crystallographic groups where wavelet sets were recently considered \cite{Merrill}.

In each of these cases, the existence of an optimal subspace was rigorously established, and a formula for selecting the generators was derived. Notably, the generators were demonstrated to belong to the invariant subspace generated by the data. However, these methods did not guarantee that the generators possessed special properties, such as compact support or smoothness.

The aim of the present work  is to focus on the smoothness of the generators. The idea is to develop a new strategy that allows one to obtain, in all the considered cases, optimal subspaces with smooth generators.
Here the smoothness is measured as a membership to a general Paley-Wiener subspace that will be chosen carefully. 
However, there is a price to pay in terms of the error for obtaining nice generators.
In what follows, we will give a detailed description of the problem and provide the definitions needed.

Given a set of data $\F =\{f_1,\dots,f_m\}$ in a Hilbert space $\Hil$, one wishes to minimize the error 
\begin{equation}\label{error}\
\E(\F,S) = \sum_{i = 1}^m \|f_i - \Proj_S{ f_i\|^2},
\end{equation}
over all subspaces $S \subset \Hil$ in a specific class. In this section, $\Lambda$ is a full rank lattice of $\R^d$.

\subsection{Case I : Shift-Invariant Subspaces (SIS)}
In  \cite{ACHM2007} the following problem  was solved:
Given a finite set of data $\F =\{f_1,\dots,f_m\}$ in $L^2(\R^d)$ and  a positive integer $\ell$, there always exist $\ell$ functions $\Phi=\{\varphi_1,\dots,\varphi_{\ell}\}$, called \emph{generators} of the subspace
$$
S_{\Lambda}(\Phi)=\overline{\text{span}}\{t_\lambda\varphi_i:  \lambda \in \Lambda, i=1,...,\ell\}\,,
$$
such that $S_{\Lambda}(\Phi)$ is optimal in the sense that it minimizes the error $$\E(\F,S_{\Lambda}(\Phi))) = \sum_{i = 1}^m \|f_i - \Proj_{S_{\Lambda}(\Phi)} f_i\|^2$$ over all other sets $\Psi$  in $L^2(\R^d)$ of $\ell$ generators. Here $t_\lambda \varphi(x) = \varphi (x - \lambda)$ is the translation of $\varphi$ by $\lambda$. Furthermore, the optimal generators can be chosen within $S_\Lambda(\F)$.

\subsection{Case II : SIS with extra invariance}
There are shift-invariant spaces $S_{\Lambda}(\Phi)$ in $L^2(\R^d)$ that are also invariant under translations by a closed subgroup $M$ of $\R^d$ containing $\Lambda$. Denote by $\mathcal V_M$ the collection of all these subspaces that have extra-invariance by $M$. The result in Case I was extended to this case. It was proved in  \cite{CM2016} that given a finite set of data $\F \subset L^2(\R^d)$ there exists a subspace $S_{\Lambda}(\Phi) \in \mathcal V_M$  with $\Phi = \{\phi_1, \dots, \phi_\ell\}$ such that $\E(\F, S_{\Lambda}(\Phi)) \leq \E(\F, S_{\Lambda}(\Psi))$ for all $S_{\Lambda}(\Psi) \in \mathcal V_M$  with $\Psi = \{\psi_1, \dots, \psi_\ell\} \subset L^2(\R^d).$

\subsection{Case III : $\Gamma$-Invariant Subspaces}
In order to allow for rotations or additional symmetries other than translations, the problem of approximation by subspaces that are invariant by non commutative groups is considered in \cite{BCHM2020}. There, the groups are of the form $\Gamma = \Lambda \rtimes G$, where $G$ is a finite group that acts faithfully on $\R^d$ by continuous automorphisms that preserve $\Lambda$, hence including the crystallographic groups that split.
	
The result in Case I was extended to this case in order to find the best $\Gamma$-invariant space of length $\ell$ that approximates the given data. This means that we find $\Phi = \{\phi_1,\dots,\phi_\ell\}$ such that
$$
\E(\F,S_{\Gamma}(\Phi)) = \sum_{i = 1}^m \|f_i - \Proj_{S_{\Gamma}(\Phi)} f_i\|^2 \leq \E(\F,S_{\Gamma}(\Psi)) ,
$$
for all sets $\Psi$  in $L^2(\R^d)$ of $\ell$ functions. Here $\Phi$ is the set of \emph{generators} of
$$
S_{\Gamma}(\Phi) := \overline{\text{span}}\{t_\lambda R_g\phi_i):  g \in G, \lambda \in \Lambda, i=1,...,\ell\} ,
$$
where $R_g (f)(x) = f(g^{-1}x)$ denotes the action of $g\in G$ on $f\in L^2(\R^d)$. In this case, again, the optimal generators can be chosen in $S_{\Gamma}(\F)$.

\

In this paper we will address the question if it is possible to obtain similar results, but requiring  that the approximating subspace consists of smooth functions. Precisely, for each one of the cases I, II, and III above, we ask if it is possible to find the optimal generators $\Phi = \{\phi_1, \dots, \phi_\ell\}$ among those belonging to a Paley-Wiener space $PW_\Omega :=\{f\in L^2(\R^d): \hat{f} \equiv 0  \text{ on }\R^d\setminus \Omega\}$ for a fixed bounded set $\Omega$ of prescribed measure. Throughout this paper, $\wh f$ denotes de Fourier transform of $f\in L^1(\R^d)$ defined by \vspace{-.8ex}$$ \vspace{-.8ex}\wh f (\omega) = \int_{\R^d} f(x) e^{-2\pi i \langle x \,, \, \omega\rangle} \, dx,$$ which extends to an isometric isomorphism of $ L^2(\R^d)$.

For a given $\Omega \subset \R^d$, for each one of the cases I, II, and III, we prove that optimization can be performed as follows: Project the data $\F$ onto $V=PW_\Omega$ to obtain the data $\Proj_V \F =\{\Proj_V f_1, \dots, \Proj_V f_m\}$, where $\Proj_V$ denotes the orthogonal projection of $L^2(\R^d)$ onto $V,$ and then minimize $\E(\Proj_V \F,S)$ over all $S\subseteq V$.  It turns out, and this is a general fact in Hilbert spaces, that if $S$ and $V$ are subspaces of a Hilbert space $\Hil$ and $S \subseteq V$, the error of approximating $\F$ by $S$ coincides with the error of approximating $\Proj_V \F$ by $S$ plus the error of approximating $\F$ by $V$. As a consequence, the subspaces that minimize $\E(\F,S)$ over all $S\subseteq V$ coincide with the subspaces that minimize $\E(\Proj_V \F,S)$ over the same collection. (See Theorem \ref{proj} in Section \ref{proj-best}.)
Then we prove that one of the subspaces that minimize the error can always be chosen inside the invariant subspace generated by the data  $\F$. (See Proposition \ref{eu:1} in Section \ref{proj-best}.)

The results outlined in the two previous paragraphs allow us to show the existence of invariant subspaces contained in $V=PW_\Omega$ that minimize the error of approximation of the data $\F$. (See Theorems \ref{teo:B}, \ref{teo:B1}, and \ref{teo:B2} in Section \ref{Section3}.) Observe that the elements of such a minimizing subspace are analytic functions in $\R^d$ since they belong to the Paley-Wiener space $PW_\Omega.$

One could design another strategy to solve this problem. Namely, find the invariant subspace that minimizes the error of approximating $\F$, and then project this subspace onto $V=PW_\Omega.$ It turns out that this method does not always produce a subspace that minimizes the error, as demonstrated by Example \ref{example} in Section \ref{Section3}.

We further show how, for each set of data, the set $\Omega$ that define the Paley-Wiener space, can be chosen in an optimal way. In particular, for non-abelian semidirect product groups $\Gamma = \Lambda \rtimes G$, the optimal $\Omega$ for the approximation must be invariant under the action of $G$. See Section~\ref{omega-opt}.

We then analyze the role of the lattices in our approximation problem. Let $A \in GL_n(\R)$ and denote by $|A|$ its determinant. We show that the best error of approximating a set of data $\F$ by $A\Lambda$-invariant subspaces of length at most $\ell$ coincides with the best error of approximation of the dilated data $D_A \F =\{D_A f_1, \dots, D_A f_m\}$ by $\Lambda$-invariant subspaces of length at most $\ell$. (See Proposition \ref{Lattices} in Section \ref{lattices}.) Here $D_A f(x) = |A|^\frac12 f(Ax)$.

We next compare the best error of approximation of the same set of data $\F$ using different lattices. Denote by $\E^*(\F, \Lambda, \ell)$ the best error of approximation of $\F$ by $\Lambda$-invariant subspaces of $\R^d$ of length less than or equal to $\ell$.  We prove that $\E^*(\F, \Lambda, \ell) \leq \E^*(\F, \Z^d, \ell)$ when $\Lambda$ is a lattice of $\R^d$ containing $\Z^d$. We show that this can be an equality even if $\Lambda$ strictly contains $\Z^d$, while the reverse inequality could also be true for a lattice $\Lambda$ that is finer than $\Z^d$ but does not contain it. (See Examples \ref{example1}, \ref{example2} and \ref{example3}, Section 6.) 

We finally analyze how the length of $S_\Lambda(\F)$ affects the approximation error. 
We exhibit a set of data $\F$ and two latices $\Lambda_1$ and $\Lambda_2$ such that the length of $S_{\Lambda_1}(\F)$ is smaller than the length of $S_{\Lambda_2}(\F)$ and also the approximation error with $\ell$ functions is smaller for $\Lambda_1$ than for $\Lambda_2.$ (See Example \ref{example4} in Section 6.)
This example induces to think that the smaller the length of $S_{\Lambda}(\F)$, the smaller the best error of approximation will be. However, Example \ref{example5} disproves that.

\

\noindent
{\bf Acknowledgments}: This project has received funding from the European Union's Horizon 2020 research and innovation programme under the Marie Sklodowska-Curie Grant Agreement No 777822. In addition, D. Barbieri and E. Hern\'andez were
supported by Grants PID2019-105599GB-I00, PID2022-142202NB-I00 / AEI / 10.13039/501100011033.
The research of C. Cabrelli and U. Molter is partially supported by Grants
PICT 2018-003399 and 2022-4875 (ANPCyT), 
PIP 11220150100355 and 202287/22 (CONICET),
UBACyT 20020170100430BA and 20020220300154BA (UBA).

\section{Projection theorems} \label{proj-best}

By a subspace $V$ of a Hilbert space, we will mean a {\it closed subspace}, and by $\Proj_V$ we denote the orthogonal projection onto $V$.

We start with the following theorem.
\begin{theo}\label{proj}
Let $\Hil$ be a Hilbert space, $\F =\{f_1,\dots,f_m\} \subset \Hil$ and let $V \subseteq \Hil$ be a fixed subspace.
Given any subspace $S \subseteq V$  we have 
\begin{equation}\label{uno-0}
\E(\F,S)=\E(\Proj_V\F,S) + \E(\F,V).
\end{equation}
As a consequence, for all subspaces $S, S' \subset V$,
$$
\E(\F,S) \leq \E(\F,S') \iff \E(\Proj_V\F,S) \leq \E(\Proj_V\F,S').
$$
\end{theo}

\begin{proof}
Suppose $S \subset V$. Then $\Proj_V \Proj_S = \Proj_S = \Proj_S \Proj_V$ and $\Proj_{V^\perp} \Proj_S = 0$. Hence, for every $f \in \Hil$
\begin{align*}
\|f - \Proj_S f\|^2 & = \|\Proj_V (f - \Proj_S f)\|^2 + \|\Proj_{V^\perp} (f - \Proj_S f)\|^2\\
& = \|\Proj_V f - \Proj_S \Proj_V f\|^2 + \|\Proj_{V^\perp} f \|^2.
\end{align*}
Thus, $\E(\F,S) = \displaystyle\sum_{i = 1}^m \|f_i - \Proj_{V} f_i \|^2 + \E({\Proj_V\F},S)$, which implies \eqref{uno-0}. The last part follow immediately from \eqref{uno-0}.
\end{proof}

We will need sometimes the following elementary Lemma about a general fact on Hilbert spaces.
\begin{lem}\label{op-commute}
	Let $V$ be a subspace of a Hilbert space $\Hil$ and let $U$ be a unitary operator in $\Hil$. Then
	
	i) $U \Proj_V =  \Proj_{U(V)} U$
	
	ii) $\Proj_V U = U\Proj_V \iff U (V) = V.$
\end{lem}
\begin{proof}
	i) The operator $T:= U \Proj_V U^*$ satisfies $T^2= T$ and $T^*=T$. Therefore, $T$ is an orthogonal projection. The result follows since the range of $T$ is $U(V).$
	
	ii) Using i) and that $U$ is an isometry we obtain:
\[
\Proj_V U = U\Proj_V \iff \Proj_V U = \Proj_{U(V)}U  \iff \Proj_V =\Proj_{U(V)} \iff V= U(V). \qedhere
\]
\end{proof}

If $\Hil$ is infinite dimensional, and $\F$ is a finite set of data, in \cite{ACHM2007} the following proposition is proved. 
\begin{prop}
Let $\F=\{f_1,...,f_m\} \subseteq \Hil$. For an arbitrary subspace $M\subseteq \Hil$, there always exists a subspace $W\subseteq \text{span}\{\F\}$ such that 
dim$(W) \leq \text{ dim}(M),$ and
\begin{equation}\label{dos}
\E(\F,W) \leq \E(\F, M).
\end{equation}
\end{prop}
We need here an extension of this property to the case that we want to approximate our data with shift, or $\Gamma$-invariant subspaces. 
Here $\Gamma =  \Lambda \rtimes G$ is the semidirect product of a full rank lattice $\Lambda$ in $\RR$ and a discrete and countable group $G$ that acts  on $\RR$ by continuous invertible automorphisms.
In addition, we also assume that $g\Lambda = \Lambda$ for all
$g\in G$, which implies that the Lebesgue measure of $\RR$ is invariant under the action of $G$.
Since shift-invariant spaces are in particular $\Gamma$-invariant for $\Gamma$ the group of translations on $\Lambda$, we state (and prove) the proposition for that case.

\begin{prop} \label{eu:1} Let $\Gamma$ be $\Gamma =  \Lambda \rtimes G$, and let $\Co_{\ell}(\R^d)$ be the class of $\Gamma$-invariant subspaces of $L^2(\R^d)$ of length at most $\ell$.
Let $\F=\{f_1,...,f_m\} \subseteq L^2(\R^d)$. For an arbitrary subspace $M\in \Co_{\ell}(\R^d)$, there always exists a subspace $W\subseteq S_{\Gamma}(\F)$ and $W\in  \Co_{\ell}(\R^d)$ such that 
\begin{equation}\label{dos2}
\E(\F,W) \leq \E(\F, M).
\end{equation}
\end{prop}
\begin{rem} The above Proposition shows that the $\Gamma$-invariant subspace of $L^2(\R^d)$ generated by at most $\ell$ elements that best approximates the data $\F$, can be taken to be generated by elements of $S_{\Gamma}(\F)$.
\end{rem}
\begin{proof}[{\bf Proof of Proposition \ref{eu:1}:}]
Let $X  = S_{\Gamma}(\F):= \ol{\vsp}\{\tr{k}\rot{g}f \, : \, k \in \Lambda, g \in G, f \in \F\}$. Let $\Psi= \{\psi_1, \dots, \psi_r\}, (r \leq \ell)$ be such that $M = S_{\Gamma}(\Psi)$. Define $\varphi_i = \Proj_X\psi_i, i=1\dots r$, and $\Phi := \{\varphi_1,\dots, \varphi_r\}$.  
For $j = 1, \dots, m$
\begin{equation}
\|f_j - \Proj_{S_{\Gamma}(\Phi)}f_j\| = \inf_{g \in S_{\Gamma}(\Phi)}\|f - g\|. \label{eu-1}
\end{equation}
Also note that since $X$ is $\Gamma$-invariant, $\Proj_X (\tr{k}\rot{g}) = (\tr{k}\rot{g})\Proj_X, \forall (k,g) \in \Gamma$ as a consequence of Lemma \ref{op-commute}, and hence $\Proj_X \Proj_{S_{\Gamma}(\Psi)}f_j \in S_{\Gamma}(\Phi) \ \forall \ j=1,\dots, m$.

Therefore, since $f_j \in \F$, by equation~\eqref{eu-1},
$$ \|f_j - \Proj_{S_{\Gamma}(\Phi)}f_j\|\leq \|f_j - \Proj_X\Proj_{S_{\Gamma}(\Psi)}f_j\| = \|\Proj_Xf_j - \Proj_X\Proj_{S_{\Gamma}(\Psi)}f_j\| \leq \|f_j - \Proj_{S_{\Gamma}(\Psi)}f_j\|.$$
Thus, with $W = S_{\Gamma}(\Phi) \in \Co_{\ell}$,
\begin{align*}
\E(\F,W) & = \E(\F,S_{\Gamma}(\Phi)) = \sum_{j=1}^m \|f_j - \Proj_{S_{\Gamma}(\Phi)}f_j\|^2\\ &\leq \sum_{j=1}^m \|f_j - \Proj_{S_{\Gamma}(\Psi)}f_j\|^2 = \E(\F,S_{\Gamma}(\Psi)) = \E(\F,M). \qedhere
\end{align*}
\end{proof}

We conclude this section by observing that enlarging the search space for the minimum with directions that are orthogonal to the data does not reduce the error of approximation.
\begin{prop}\label{prop:useless}  Let $S, V,$ and $W$ be subspaces of $\Hil$. Suppose $V \perp W,\  S\perp W$  and let $S' = S \oplus W$. Then
$$\E(\Proj_V\F,S)=\E({\Proj_V\F},S').$$
\end{prop}

\begin{proof}
\begin{align*}
\E({\Proj_V\F},S') & = \sum_{i = 1}^m \|\Proj_V f_i - \Proj_{S'} \Proj_V f_i\|^2 = \sum_{i = 1}^m \|\Proj_V f_i - (\Proj_{S} + \Proj_{W}) \Proj_V f_i\|^2\\
& = \sum_{i = 1}^m \|\Proj_V f_i - \Proj_{S} \Proj_V f_i\|^2 = \E({\Proj_V\F},S). \qedhere
\end{align*}
\end{proof}

\section{Smooth generators} \label{Section3}

\subsection{Approximation by Shift-Invariant Spaces} In this subsection we will consider the case in which we want to approximate our data with a shift-invariant space generated by  ''smooth '' functions. For this, let $\Omega \subset \R^d$ be a bounded measurable set. We will later see how to choose a convenient set $\Omega$ depending of the data $\F.$
Throughout this section we will fix a full rank lattice $\Lambda$ of $\R^d$ and  the shift-invariant spaces are invariant along translations by $\Lambda$. Thus we will omit the subscript $\Lambda$, and simply write $S(\Phi)$ instead of $S_\Lambda(\Phi).$

Let $V = PW_{\Omega}$, be the generalized Paley-Wiener space associated to $\Omega.$ i.e. $PW_\Omega=\{f\in L^2(\R^d): \hat{f} \equiv 0  \text{ on }\R^d\setminus \Omega\}.$
For a positive integer $\ell$, we denote by $\Co_\ell(\R^d) $ the class of {\it shift-invariant subspaces} of length at most $\ell$, 
and by $\Co_\ell(\Omega)$ the elements of $\Co_\ell(\R^d)$ included in $V$. We recall that the {\it length} of a shift-invariant subspace is the minimal number of generators of the subspace.

In \cite{ACHM2007} the existence of a best shift-invariant space for given data was proven. If we want to apply the solution of that paper, we cannot guarantee that the generators of the optimal subspace are smooth, i.e. that they belong to $V$.
However, using Theorem~\ref{proj}, we will see that, if instead we solve the theorem of approximation for the data $\Proj_V \F$, instead of $\F$, this yields the solution sought after. That is, the optimal subspace for the data $\Proj_V \F$ is included in $V$
and minimizes the error. 

We have the following Theorem.
\begin{theo}\label{teo:B}
Given data $\F = \{f_1, \dots, f_m\} \subset L^2(\R^d)$, and $\Omega \subset \R^d$, there exists a shift-invariant space 
$S(\Phi)\in \Co_\ell(\Omega)$ such that $\E(\F,S(\Phi)) \leq \E(\F,S(\Psi))$ for all sets $\Psi=\{\psi_1,\dots,\psi_\ell\} \subset PW_{\Omega}$.
\end{theo}

\begin{proof}[Proof of Theorem~\ref{teo:B}]
Since any Paley Wiener space is translation invariant for any translation, part $ii)$ of Lemma~\ref{op-commute} shows that $\Proj_V$ commutes with any translation, which in particular implies that the orthogonal projection of any subspace $S(\Psi)$ of 
$\Co_\ell(\R^d)$  onto $V$ is in $\Co_\ell(\Omega)$ and $\Proj_V S(\Psi)=S(\Proj_V(\Psi))$.
	
Let now $S^*=S^*(\Phi)$ be an optimal subspace that approximates the projected data  $\Proj_V\F$ with elements of $\Co_\ell(\R^d)$. By Proposition ~\ref{eu:1} we can choose $S^*\subseteq S (\Proj_V\F) = \Proj_V S(\F) \subseteq V$. Since $S^*$ minimizes $\E(\Proj_V\F,S)$ over all $S \in \Co_\ell(\R^d)$, then in particular
$S^*$ minimizes $\E(\Proj_V\F,S)$ over all $S \in \Co_\ell(\Omega).$
That is,
$$
\E(\Proj_V \F, S^*) = \underset{S \in \Co_\ell(\R^d)} {\operatorname{min}} \, \E(\Proj_V F,S) = \underset{S \in \Co_\ell(\Omega)} {\operatorname{min}} \, \E(\Proj_V F,S).
$$
Therefore, by Theorem~\ref{proj},
\begin{align*}
	\E(\F,S^*) &= \E(\Proj_V \F,S^*) + \E(\F,V)\\
	&= \underset{S \in \Co_\ell(\Omega)} {\operatorname{min}} \, \E(\Proj_V F,S) + \E(\F,V)
	= \underset{S \in \Co_\ell(\Omega)} {\operatorname{min}} \E(\F,S)\,,
\end{align*}
which proves the result.
\end{proof}

\begin{rem}
For a class of subspaces $\Co$ of $L^2(\R^d)$ denote by $ \displaystyle \underset{S \in \Co} {\operatorname{argmin}} \, \E(\F,S)$ the set of all subspaces $A\in \Co$ such that $\displaystyle \E(\F,A) = \underset{S \in \Co} {\operatorname{min}} \, \E(\F,S)$. Then we show the following:
\begin{equation}\label{eq:minimizing2}
	\Big(\underset{S \in \Co_{\ell}(\R^d)} {\operatorname{argmin}} \,\E(\Proj_V \F,S)\Big) \cap \Co_\ell(\Omega)  =
	\underset{S \in \Co_\ell(\Omega)} {\operatorname{argmin}} \, \E(\Proj_V \F,S) = 
	\underset{S \in \Co_\ell(\Omega)} {\operatorname{argmin}} \, \E(\F,S)\,.
\end{equation}
The second identity holds by the second part of Theorem \ref{proj}. For the first identity, we first prove the inclusion $\subseteq$. If $A$ is in the left hand set of \eqref{eq:minimizing2}, then $A\in \Co_\ell(\Omega)$ and 
$$ 
\E(\Proj_V \F, A) = \underset{S \in \Co_{\ell}(\R^d)} {\operatorname{min}} \,\E(\Proj_V \F,S) \leq \underset{S \in \Co_{\ell}(\Omega)} {\operatorname{min}} \,\E(\Proj_V \F,S)  \leq \E(\Proj_V \F, A).
$$
This shows that $\displaystyle A\in \underset{S \in \Co_\ell(\Omega)} {\operatorname{argmin}}\, \E(\Proj_V \F,S)$.

Let us now prove the opposite inclusion, that is, let us prove that if $S^* \in \underset{S \in \Co_{\ell}(\Omega)} {\operatorname{argmin}} \,\E(\Proj_V \F,S) $ then $S^* \in \Big(\underset{S \in \Co_{\ell}(\R^d)} {\operatorname{argmin}} \,\E(\Proj_V \F,S)\Big) \cap \Co_\ell(\Omega)$. To this end assume, by contradiction, that there is an $\tilde{S} \in \Co_{\ell}(\R^d)$ such that $\E(\Proj_V \F, \tilde{S}) < \E(\Proj_V\F,S^*)$. Thus, by Proposition \ref{eu:1}, there is a $W \in \Co_{\ell}(\Omega)$ such that $\E(\Proj_V \F, W) \leq \E(\Proj_V\F,\tilde{S}) < \E(\Proj_V\F,S^*)$, but by definition of $S^*$ we have that $\E(\Proj_V \F, S^*) \leq \E(\Proj_V \F, W)$.

Note that the first identity in \eqref{eq:minimizing2} would not hold in general if we remove from the left hand side the intersection with $\Co_\ell(\Omega)$, as it can be easily seen using Proposition \ref{prop:useless}.
\end{rem}

At this point one can ask what would happen if we find and optimal subspace $S^1 \in C_\ell(\R^d)$ for approximating the data $\F$ with subspaces of $C_\ell(\R^d)$ and then project this subspace onto $V$ to obtain a subspace $S^1_V= \Proj_V (S^1).$ 

Let $S^*=S(\Phi)$ be an optimal subspace for approximating the modified data $\Proj_V \F$ with subspaces that belong to $C_\ell(\R^d).$ By Proposition \ref{eu:1}, we can choose $S^*\subseteq S (\Proj_V\F) = \Proj_V S(\F) \subseteq V.$ Recall that $V=PW_\Omega.$ Since $S^* \subseteq V$, by Theorem \ref{proj},
$$
\E(\F, S^*)= \E(\Proj_V\F, S^*) + \E(\F, V)\,.
$$
Since $S^1_V \subseteq V$, we can see that
\begin{equation} \label{eq8}
	\E(\F,S^*)=\E(\Proj_V\F,S^*) + \E(\F,V) \leq \E(\Proj_V\F,S^1_V) + \E(\F,V) = \E(\F,S^1_V).
\end{equation}
In some cases, we could have equality in \eqref
{eq8}. For example, this happens if all the data $\F$ is in $V$ or in $V^\perp$. On the other hand, we will show in the next example that the inequality in \eqref{eq8} may be strict.

\begin{example}{\phantom{\!\!\!}}{}\footnote{We thank In\'es Armendariz for the idea behind this example.} \label{example}
Let $\F=\{ f_1, f_2 \}$ with
$$
\wh {f_1} = \chi_{[-1,0)} + 2 \chi_{[1,2)}\,, \qquad \qquad \wh{f_2} = -\chi_{[-1,0)} + 2 \chi_{[1,2)}\,.
$$
Let $V=PW_{[-1,1)}.$ We compute first the right hand side of \eqref{eq8} with $\Z$-invariant subspaces of length 1 in $L^2(\R).$ To find $S^1 \in \Co_1(\R)$ we need to find $\varphi_0 \in L^2(\R)$ such that
$$
\E(\F,S(\varphi_0)) = \underset{\varphi \in L^2(\R)} {\operatorname{min}} \, \E(\F,S(\varphi))\,.
$$
By Plancherel Theorem,
\begin{equation} \label{X1}
	\E(\F,S(\varphi))= \|\wh{f_1}- \wh{\Proj_{S(\varphi)}f_1}\|^2 + \|\wh{f_2}- \wh{\Proj_{S(\varphi)}f_2}\|^2\,.
\end{equation}
It can be seen, after a moment thought, that for $\varphi$ to minimize $\eqref{X1}$, ${\operatorname{supp} \wh {\varphi}}$ must be contained in $[-1,0)\cup [1,2).$ Among such functions, the smallest error is attained for $\varphi_0$ given by $\wh {\varphi_0} = \chi_{[1,2)}.$ Observe that $\wh {\Proj_{S(\varphi_0)} f_1} = 2\chi_{[1,2)}= \wh {\Proj_{S(\varphi_0)} f_2} = 2\chi_{[1,2)}$. Therefore, using \eqref{X1} we conclude 
$$
\E(\F, S(\varphi_0))= \|\chi_{[-1,1)}\|^2 +\|-\chi_{[-1,1)}\|^2 = 2.
$$
We have shown that $S^1 = S(\varphi_0)$. Hence, $S^1_V = \Proj_V(S^1) = \Proj_V (S(\varphi_0)) = \{ 0\}$, since ${\operatorname{supp} \wh {\varphi_0}} \cap [-1,1) = \emptyset.$ In this case, the right hand side of \eqref{eq8} becomes
$$
\E(\F, S^1_V) = \|f_1\|^2 + \|f_2\|^2 =  10.
$$
We compute now the left hand side of \eqref{eq8}. First find $S^*= S(\varphi_1)$ such that 
$$
\E(\Proj_V\F,S^*) = \underset{S(\varphi) \subset V} {\operatorname{min}} \, \E(\Proj_V \F,S(\varphi))\,.
$$
By taking $\wh {\varphi_1} = \chi_{[-1,0)}$ (observe that $\varphi_1 \in V = PW_{[-1,1)}),$
\begin{align*} 
	\E(\Proj_V \F,S(\varphi_1)) &= \|\wh{\Proj_V f_1}- \left(\Proj_{S(\varphi_1)} \Proj_V f_1\right)\wh{\ }\|^2 + \|\wh{\Proj_V f_2}- \left(\Proj_{S(\varphi_1)} \Proj_V f_2\right)\wh{\ }\|^2\,\\
	& = \|\chi_{[-1,0)} - \chi_{[-1,0)}\|^2 + \| - \chi_{[-1,0)} + \chi_{[-1,0)}\|^2 = 0\,.
\end{align*}
This shows that $S^* = S(\varphi_1).$ By Theorem \ref{proj}, the left hand side of \eqref{eq8} becomes 
\begin{align*} 
	\E(\F,S^*) &= \E(\F,V) = \|\wh{f_1}- \wh{\Proj_V  f_1}\|^2 + \|\wh{f_2}- \wh{\Proj_{V}f_2}\|^2 \\
	& = \|2\chi_{[1,2)}\|^2 + \| 2\chi_{[1,2)} \|^2 = 8 \,.
\end{align*}
\end{example}

\subsection{Approximation by Shift-Invariant Spaces with Extra invariance} 
There are shift-invariant spaces $S_\Lambda(\Psi)$ in $L^2(\R^d)$ that are also invariant under translations by the elements of a closed subgroup $M$ of $\R^d$ containing $\Lambda.$ In this case we say that $S_\Lambda(\Psi)$ has extra-invariance. Denote by $\mathcal V_M$ the set of such subspaces of $L^2(\R^d)$.

Since in Theorem 3.1 of \cite{CM2016} the existence of ''best shift-invariant spaces that have extra-invariance'' has been proven, if we want to require that the approximating vectors are in  $\Co_\ell(\Omega)$ we can apply the same proof of the previous theorem, since again, $\Proj_V S_\Lambda(\Phi)=S_\Lambda(\Proj_V(\Phi))$. Therefore we have the following theorem.
\begin{theo}\label{teo:B1}
Let $M$ be closed subgroup of $\R^d$ containing $\Lambda.$
Given data $\F = \{f_1, \dots, f_m\} \subset L^2(\R^d)$, and $\Omega \subset \R^d$, there exists a shift-invariant space $S_\Lambda(\Phi) \in \mathcal V_M  \cap \Co_\ell(\Omega)$ such that $\E(\F,S_\Lambda(\Phi)) \leq \E(\F,S_\Lambda(\Psi))$ for all $S_\Lambda(\Psi)\in \mathcal V_M$ and $\Psi=\{\psi_1,\dots,\psi_\ell\} \subset PW_{\Omega}$.
\end{theo}

\subsection{Approximation by $\Gamma$-Invariant Spaces}  
In this section $\Gamma = \Lambda \rtimes G$ where $\Lambda$ is a full rank lattice of $\R^d$ and $G$ is a finite group that acts faithfully on $\R^d$ by continuous automorphisms that preserve $\Lambda.$
In this case the situation is slightly different, since the proof of an analogue of Theorem~\ref{teo:B} has to rely on the fact that $\Proj_V S_{\Gamma}(\Phi)=S_{\Gamma}(\Proj_V(\Phi))$. For an arbitrary set $\Omega$ this will not necessarily be true. By Lemma~\ref{op-commute}, the condition we need to impose on $\Omega$ is  that it is invariant under the action of the group $G$, that is $g\Omega = \Omega, \forall g \in G$.

With this extra condition we have again $\Proj_V S_{\Gamma}(\Phi)=S_{\Gamma}(\Proj_V(\Phi))$. Now, using that in Theorem 5.2 of \cite{BCHM2020} the existence of ''best  $\Gamma$-invariant spaces'' has been proven, if we want to require that the approximating vectors are in  $\Co_\ell(\Omega)$ we can apply the same proof of Theorem~\ref{teo:B}, to obtain, the following Theorem.
\begin{theo}\label{teo:B2}
Given data $\F = \{f_1, \dots, f_m\} \subset L^2(\R^d)$, $\Gamma = \Z^d \rtimes G$, with $G$ being a finite group of that acts faithfully on $\R^d$ by continuous automorphisms that preserve $\Z^d$, and $\Omega \subset \R^d$ such that $\Omega$ is invariant under the action of $G$, there exists a $\Gamma$-invariant space $S_{\Gamma}(\Phi)$, such that
$S_{\Gamma}(\Phi)\in \Co_\ell(\Omega)$ and $\E(\F,S_{\Gamma}(\Phi)) \leq \E(\F,S_{\Gamma}(\Psi))$ for all sets $\Psi=\{\psi_1,\dots,\psi_\ell\} \subset PW_{\Omega}$.
\end{theo}

\section{Best set $\Omega$}  \label{omega-opt}

\subsection{Shift-Invariant case}
Since we are looking to find a compact set that {\em best} fits the data in $\R^d$, in the sense that the Fourier Transform of the data is inside the set, without loss of generality, we can assume that the data already is in a huge $PW_{[-N,N]^d}$.

\begin{lem}\label{lem:bestPW}
Let $\varphi : \R^d \to \R_{\geq 0}$ be a measurable function, and denote by $\Omega$ the function from $\R_{\geq 0}$ to the measurable sets of $\R^d$ given by
$$
\Omega(t) = \left\{\xi \in \R^d : \varphi (\xi) > t\right\}.
$$
For any $M > 0$, set $t^M = \sup\{t : |\Omega(t)| > M\}$ and $t_M = \inf\{t : |\Omega(t)| < M\}$.
Let $\Omega_M \subset \R^d$ be any measurable set of measure $|\Omega_M| = M$ satisfying
\begin{equation}\label{eq:inclusions}
\Omega(t_M) \subset \Omega_M \subset \Omega(t^M).
\end{equation}
Then
$$
\max_{E \subset \R^d \,:\, |E| = M} \int_{E} \varphi(\xi) d\xi = \int_{\Omega_M} \varphi(\xi) d\xi.
$$
\end{lem}
\begin{proof}
Observe first that, since $t \mapsto |\Omega(t)|$ is a monotone decreasing function which goes to zero at infinity, both $t^M$ and $t_M$ exist and are finite for every $M > 0$, with $t^M \leq t_M$. 
Assume now, by contradiction, that for a given $M > 0$ there exists a measurable $E \subset \R^d$, with $|E| = M$, such that
$$
\int_{E} \varphi(\xi) d\xi > \int_{\Omega_M} \varphi(\xi) d\xi.
$$
Denoting by $E_0 = E \setminus (E \cap \Omega_M)$ and by $\Omega_0 = \Omega_M \setminus (\Omega_M \cap E)$, this reads
$$
\int_{E_0} \varphi(\xi) d\xi > \int_{\Omega_0} \varphi(\xi) d\xi
$$
where $|E_0| = |\Omega_0|$ and $E_0 \cap \Omega_0 = \emptyset$. This implies in particular that there is a subset $E_+ \subset E_0$ of positive measure such that
\begin{equation}\label{eq:contradiction}
\inf_{E_+} \varphi > \varphi(\xi) \textnormal{ for a.e. } \xi \in \Omega_0. 
\end{equation}
Now, since by the right hand side inclusion of \eqref{eq:inclusions} we have $\varphi(\xi) > t^M$ for a.e. $\xi \in \Omega_0$, we have obtained that $\varphi(\xi) > t^M$ for a.e. $\xi \in E_+$, that is, $E_+ \subset \Omega(t^M)$. However, by the left hand side inclusion of \eqref{eq:inclusions} we have that the largest values in $\Omega(t^M)$ are already attained in $\Omega_0$, hence contradicting \eqref{eq:contradiction}.
\end{proof}

\begin{theo}\label{bestPW2}
Let $\F = \{f_i\}_{i = 1}^m \subset L^2(\R^d)$, let $\displaystyle \varphi(\xi) = \sum_{i = 1}^m |\wh{f}_i(\xi)|^2$ and, for any $M > 0$, let $\Omega_M$ be as in \eqref{eq:inclusions}.
Then
\begin{equation}\label{eq:bestPW}
\min_{E \subset \R^d \,:\, |E| = M} \sum_{i = 1}^m \|\wh{f_i}(1 - \chi_E)\|^2 = \sum_{i = 1}^m \|\wh{f_i}(1 - \chi_{\Omega_M})\|^2. 
\end{equation}
\end{theo}
\begin{proof}
The identity \eqref{eq:bestPW} reads
$$
\sum_{i = 1}^m \int_{\R^d} |\wh{f_i}(\xi)|^2 \, |1 - \chi_{\Omega_M}(\xi)|^2 d\xi \leq \sum_{i = 1}^m \int_{\R^d} |\wh{f_i}(\xi)|^2 \, |1 - \chi_{E}(\xi)|^2 d\xi
$$
for all $E$ of measure $M$. That is, using that $|1 - \chi_{E}(\xi)|^2 = 1 - \chi_{E}(\xi)$,
$$
\int_{\Omega_M} \left(\sum_{i = 1}^m |\wh{f_i}(\xi)|^2\right) d\xi \geq \int_{E} \left(\sum_{i = 1}^m |\wh{f_i}(\xi)|^2\right) d\xi \quad \forall \, |E| = M.
$$
Hence, the conclusion follows by Lemma \ref{lem:bestPW}
\end{proof}

\begin{rem}
Lemma \ref{lem:bestPW} and Theorem \ref{bestPW2} are still valid if $\R^d$ is replaced by a $\sigma$-finite measure space $(X,\mu)$.
Furthermore,  Theorem \ref{bestPW2} also shows that  actually there exist a set $E_0$, with  $|E_0| = M$  that maximize  $\{\int_{E} \varphi :E \subset \R^d \,:\, |E| = M\} $.
\end{rem}
 
 \subsection{$\Gamma-$invariance case}\label{omega-opt-gamma}

When we optimise with  $\Gamma-$invariant subspaces, we require  the Paley-Wiener space $PW_\Omega$ to be $\Gamma$-invariant.
Since any Paley-Wiener space is translation invariant and $\Gamma = \Z^d \rtimes G,$ we only need to require  $PW_\Omega$ to be $G-$invariante that is,
$\Omega$ to be invariant under the action of $G$ (i.e. $g\,\Omega=\Omega $ for all $ g\in G$).

Choose a positive real number $M$ and denote by $\M(G,M)=\{\Omega\subset \R^d:|\Omega| = M, \text{ and }g\,\Omega=\Omega, \; \forall  g\in G\}$.
Thus, given any $M > 0$ we need
 to find a measurable set $\Omega\in \M(G,M) $ that  maximizes $\int_\Omega \varphi $ over the class $\M(G,M)$. 

For this,
 let $\F = \{f_i\}_{i = 1}^m \subset L^2(\R^d)$, 
\begin{equation}\label{eq:smallphi}
\varphi(\xi) = \sum_{i = 1}^m |\wh{f}_i(\xi)|^2
\end{equation}
and 
\begin{equation}\label{eq:bigphi}
\Phi_G(\xi) = \sum_{g\in G} \sum_{i = 1}^m |\wh{f}_i(g\xi)|^2.
\end{equation}
Define the equivalence relation $\sim_G$  in $\R^d$  by $x \sim_G y$ if and only if there exist $g \in G$ such that $x=gy$ .
Choose $D\subset \R^d$ to be a measurable section of the quotient  $\R^d/{\sim_G}$.  Thus, $\R^d = \bigcup_{g\in G} g D$, with the union being disjoint.

The existence of a measurable section $D$ can be obtained in the following way: In Proposition 4.2 in \cite{BCHM2020} it is proved that there exists a measurable set $P \subset \R^d$ of finite and positive Lebesgue measure such that, up to a measure zero set, 
$$\R^d = \bigcup_{k \in \Z^d} \Big(\bigcup_{g \in G} gP \Big) + k,$$
with the union being disjoint.
Using this we have,
$$\R^d = \bigcup_{k \in \Z^d} \Big(\bigcup_{g \in G} gP \Big) + k = \bigcup_{k \in \Z^d} \bigcup_{g \in G} g(P +  g^{-1}k) = \bigcup_{g\in G}g \left(\bigcup_{k \in \Z^d}(P + g^{-1}k)\right)
. $$
Since $g\Z^d = \Z^d$ for all $g \in G$, we have $\bigcup_{k \in \Z^d}(P + g^{-1}k)= \bigcup_{k \in \Z^d}(P + k)$ and so $\R^d = \bigcup_{g\in G}g D$ with $D = \bigcup_{k \in \Z^d}(P + k)$ which is clearly measurable.

\begin{theo} Let $G$ be a finite group that leaves $\Z^d$ invariant. Let $M > 0$ and $\M \equiv \M(G,M) = \{ \Omega \subset \R^d: |\Omega|=M \  \mbox{and} \ \ g\Omega= \Omega \ \ \forall g \in G\}$. For $\varphi, \Phi_G$ as in \eqref{eq:smallphi} and \eqref{eq:bigphi}, assume that
$$
\Omega_0 \in A =  {\operatorname{argmax}} \left\{ \int _\Omega \varphi \ : \ \Omega \in \M \right\}
$$
and
$$
\Sigma_0 \in B = \operatorname{argmax} \left\{\int_{\Sigma} \Phi_G:\Sigma \subseteq D, |\Sigma|=M/(\#G)\right \}. $$   Then

$$\int_{\Omega_0}\varphi  = \int_{\Sigma_0} \Phi_G.$$
\end{theo}

\begin{proof}
	For $\Omega_0\in A$, consider $\Sigma = \Omega_0 \cap D$. Since $\Omega_0 \in \M$,
	\begin{equation} \label{Omega}
	\Omega_0 = ( \bigcup_{g\in G} gD )\cap \Omega_0 = \bigcup_{g\in G} (gD \cap \Omega_0) = \bigcup_{g\in G} (gD \cap g\Omega_0) =  \bigcup_{g\in G} g \Sigma\,.
	\end{equation}
    Therefore, 
	\begin{equation}\label{integrals}
	\int_{\Omega_0} \varphi(\xi)\, d\xi = \int_{\underset{g\in G} {\operatorname{\cup}}g\Sigma} \; \varphi(\xi)\, d\xi = 
	\sum_{g\in G}\int_{\Sigma} \varphi(g\omega) \,d\omega =  \int_{\Sigma} \Phi_G (\omega) \, d\omega.
	\end{equation}
	Since $G$ preserves the lattice $\Z^d$ it also preserves the Lebesgue measure of $\R^d$. From \eqref{Omega} we obtain $|\Sigma|= |\Omega_0|/(\#G) = M/(\#G).$ Hence, as $\Sigma_0 \in B$, we have
	$$
	\int_{\Omega_0} \varphi(\xi)\, d\xi = \int_{\Sigma} \Phi_G (\omega) \, d\omega \leq \int_{\Sigma_0} \Phi_G (\omega) \, d\omega.
	$$
	
	Now, for $\Sigma_0 \in B$, define $\displaystyle \Omega= \bigcup_{g\in G} g\Sigma_0$. The same computation as in \eqref{integrals} gives 
	$$
	\int_{\Sigma_0} \Phi_G (\omega) \, d\omega = \int_{\Omega} \varphi(\xi)\, d\xi \,.
	$$
	Since $|\Omega| = (\#G) |\Sigma_0| = M$ with $g\Omega = \Omega$ for all $g\in G$, and as $\Omega_0 \in A$, we conclude
\[
	\int_{\Sigma_0} \Phi_G (\omega) \, d\omega = \int_{\Omega} \varphi(\xi)\, d\xi \leq  \int_{\Omega_0} \varphi(\xi)\, d\xi\,. \qedhere
\]
\end{proof}

\section{Lattices} \label{lattices}

Given a full rank lattice $\Lambda$ in $\R^d$, let $\Omega_\Lambda$ be a fundamental set of $\Lambda^\perp$ (the annihilator of $\Lambda$), that is a set such that $\Omega_\Lambda \approx \wh{\R^d} / \Lambda^\perp$. For $f \in L^2(\R^d)$, denote by
\begin{equation} \label{ii}
\iso_\Lambda f (\omega) = \{\wh{f}(\omega + \gamma)\}_{\gamma \in \Lambda^\perp} \, , \quad \omega \in \Omega_\Lambda\,,
\end{equation}
defining an isometric isomorphism $\iso_\Lambda : L^2(\R^d) \to L^2(\Omega_\Lambda,\ell_2(\Lambda^\perp))$ (see Proposition 1.2 in \cite{Bow2000} for the case $\Lambda = \Z^d$ and Proposition 3.3 in \cite{CP2010} for the general case). 

For a family $\F = \{f_i\}_{i = 1}^m \subset L^2(\R^d)$, denote also its Gramian matrix with respect to $\Lambda$ by
$$
(\G^\F_\Lambda)_{ij}(\omega) = \sum_{\gamma \in \Lambda^\perp} \wh{f_i}(\omega + \gamma) \ol{\wh{f_j}(\omega + \gamma)} \, , \quad \omega \in \Omega_\Lambda.
$$
We have the following elementary lemma whose proof is left to the reader. 
\begin{lem}\label{lem:AF}
Let $A \in \textrm{GL}_d(\R)$ and denote by $\wh{A} = (A^t)^{-1}$ and by $|A|=det(A)$.
\begin{itemize}
\item[i)] Let $\Lambda \subset \R^d$ be a full rank lattice and let $\Omega_\Lambda$ be a fundamental set for $\wh{\R^d}/\Lambda^\perp$. Then $A\Lambda$ is a full rank lattice and $\wh{A}\Omega_\Lambda$ is a fundamental set for $\wh{\R^d}/\wh{A}\Lambda^\perp$.
\item[ii)] Let $f \in L^2(\R^d)$ and denote by $D_A f(x) = |A|^\frac12 f(Ax)$. Then
$$
\wh{D_A f} = D_{\wh{A}} \, \wh{f} .
$$
\end{itemize}
\end{lem}

\begin{prop}\label{dilation}
Let $\Lambda \subset \R^d$ be a full rank lattice and let $\Omega_\Lambda$ be a fundamental set for $\wh{\R^d}/\Lambda^\perp$. Let $\F = \{f_i\}_{i = 1}^m \subset L^2(\R^d)$, and denote by $D_A\F = \{D_A f_i \}_{i = 1}^m$.
\begin{itemize}
\item[a)] for all $f \in L^2(\R^d)$ we have
$$
\iso_{A\Lambda} f (\wh{A}\omega) = |A|^\frac12 \iso_\Lambda (D_A f)(\omega) \, , \quad \omega \in \Omega_\Lambda.
$$
\item[b)] the following covariance relationship for the Gramian holds:
$$
\G_{A\Lambda}^\F(\wh{A}\omega) = |A| \G_\Lambda^{D_{A}\F}(\omega) \, , \quad \omega \in \Omega_\Lambda.
$$
\end{itemize}
\end{prop}
\begin{proof}
Both parts are consequence of item $ii)$ of Lemma \ref{lem:AF}. We only proof $b).$
 For $\omega \in \Omega_\Lambda$ we have
\begin{align*}
\Big(\G_{A\Lambda}^\F(\wh{A}\omega) \Big)_{ij}
& = \sum_{\lambda \in \wh{A}\Lambda^\perp} \wh{f_i}(\wh{A}\omega + \lambda) \ol{\wh{f_j}(\wh{A}\omega + \lambda)}\\
& = \sum_{\gamma \in \Lambda} \wh{f_i}(\wh{A}(\omega + \gamma)) \ol{\wh{f_j}(\wh{A}(\omega + \gamma))}\\
& = |A| \sum_{\gamma \in \Lambda^\perp} D_{\wh{A}}\wh{f_i}(\omega + \gamma) \ol{D_{\wh{A}}\wh{f_j}(\omega + \gamma)} \\
& = |A| \sum_{\gamma \in \Lambda^\perp} \wh{D_{A}f_i}(\omega + \gamma) \ol{\wh{D_{A}f_j}(\omega + \gamma)}=  |A| \Big(\G_{\Lambda}^{D_{A}\F } (\omega) \Big)_{ij}. \qedhere
\end{align*}
\end{proof}

%

Recall that, for a general full rank lattice $\Lambda$ of $\R^d$, we denote by $\E^*(\F,\Lambda, \ell)$ the minimum error of approximation of the data $\F$ by a $\Lambda$-invariant subspace with at most $\ell$ generators.

\begin{prop}\label{Lattices}
For any full rank lattice $\Lambda$ of $\R^d$, any set of data $\F=\{f_1, \dots, f_m \} \subset L^2(\R^d)$, and any $A\in GL_n(\R)$,
$$
\E^*(\F,A\Lambda,\ell)= \E^*(D_A\F, \Lambda,\ell).
$$
\end{prop}

\begin{proof}
Let $S= S_{A\Lambda}(\{ \psi_1, \dots, \psi_\ell  \})$ be a general $A\Lambda$-invariant subspace of $L^2(\R^d)$ of length $\ell$. Let $S' =S_\Lambda (\{ D_A \psi_1, \dots, D_A\psi_\ell  \})$ be the $\Lambda$-invariant subspace generated by $\{ D_A \psi_1, \dots, D_A\psi_\ell  \}$. By part $i)$ of Lemma \ref{op-commute},
$$
D_A \Proj_S = \Proj_{D_A(S)} D_A\,.
$$
Since for any $\psi_j$, $j=1,2, \dots, \ell$, and any $\gamma \in \Lambda$,
$$
D_A T_{A\gamma}\psi_j(x) = |A|^{1/2}\psi_j(A(x-\gamma)) = T_\gamma D_A \psi_j(x)\,,
$$
we conclude that $D_A(S)=S'.$ Therefore,
$$
D_A \Proj_S = \Proj_{S'} D_A\,.
$$
It follows that,
\begin{align*}
\E(\F,S) & = \sum_{i=1}^m \| f_i- \Proj_S f_i\|^2 =
\sum_{i=1}^m \| D_A f_i- D_A\Proj_S f_i\|^2 \\
& = \sum_{i=1}^m \| D_A f_i- \Proj_{S'} D_A f_i\|^2 = \E(D_A \F, S')\,. 
\end{align*}
This shows that for any $A\Lambda$-invariant subspace $S$ of $L^2(\R^d)$ of length $\ell$ there is a $\Lambda$-invariant subspace of $L^2(\R^d)$, namely $S'$, such that $\E(\F, S) = \E(D_A \F, S')$. Therefore, $\E^*(D_A F, \Lambda, \ell) \leq \E^*(\F, A\Lambda, \ell).$  The reverse inequality is proved similarly, since any $\Lambda$-invariant subspace $V =S_\Lambda(\{ \varphi_1, \dots, \varphi_\ell\})$ of $L^2(\R^d)$ of length $\ell$ produces and $A\Lambda$ invariant subspace $V' = S_{A\Lambda} (\{ D_{A^{-1}}\varphi_1, \dots, D_{A^{-1}}\varphi_\ell \})= D_{A^{-1}}(V)$ for which $\E(D_A \F, V)= \E(F, V').$
\end{proof}

\begin{rem}
A different proof of Proposition \ref{Lattices} can be given using part b) of Proposition \ref{dilation} and Theorem 2.3 in \cite{ACHM2007} (generalized to include full-rank lattices of $\R^d$ in Theorem 5.2 in \cite{BCHM2020}).
\end{rem}

\section{Examples} \label{examples}

Let $\F=\{f_1, \dots, f_m\}$ be a finite set of data in $L^2(\R^d)$ and let $\Lambda$ be a full rank lattice of $\R^d.$ Proposition \ref{Lattices} shows that the best error of approximation of the modified data $D_A \F = \{ D_A f_1, \dots, D_A f_m \}$ by $\Lambda$-invariant subspaces of length $\ell$ coincides with the best error or approximation of the original data $\F$ by $A\Lambda$-invariant subspaces also with length $\ell$.

The situation is different if we try to compare the best error of approximation of the same data with subspaces that are invariant by two different lattices. We will show with one and two dimensional examples that the comparison depends on how the lattices relate to each other and also how the data behaves with respect to the lattices.

The simplest situation is when we consider the lattice $\Z^d$ and a full rank lattice $\Lambda$ containing $\Z^d$.
The set
$$
\{ T_\gamma \psi_i: \gamma \in \Lambda, i=1, \dots, \ell \}\,,
$$
clearly contains the set
$$
\{ T_k \psi_i:  k\in \Z^d, i=1, \dots, \ell\}\,.
$$
Therefore, if $\Psi = \{ \psi_1, \dots, \psi_\ell\},$
\begin{equation}\label{inclusion}
	S_{\Z^d}(\Psi) \subseteq S_\Lambda (\Psi).
\end{equation}
For a given function $f\in L^2(\R^d)$, the distance between $f$ and each one of the subspaces in \eqref{inclusion} is smaller for the bigger subspace. Therefore, for any data set $\F$,
$$
\E(\F, S_\Lambda(\Psi)) \leq \E(\F, S_{\Z^d}(\Psi))\,.
$$
Minimizing over all collections $\Psi$ of $\ell$ functions of $L^2(\R^d)$ we deduce
\begin{equation} \label{inequality}
	\E^* (\F, \Lambda, \ell) \leq \E^*(\F, \Z^d, \ell)\,.
\end{equation}


\subsection{One dimensional examples}
Rewriting \eqref{inequality} for $d=1$ and $N \in \N$, we have
\begin{equation} \label{Inequality2}
\E^*(\F, \Z/N, \ell) \leq \E^*(\F,\Z, \ell)\,.
\end{equation}
In this section we shall show that, if $N\geq 2$, in some situations the above inequality is strict, but in others is not, depending on the data $\F$ we want to approximate.

In the next examples we shall use the following result: 
\begin{lem}\label{characterization}
Let $\Lambda$ be a full rank lattice in $\R^d$ and $\psi \in L^2(\R^d)\,.$
\begin{itemize}
\item[a)] If $f\in S_\Lambda(\psi)$, there exists a measurable $\Lambda^\perp$-periodic function $m$ in $\R^d$ such that $\widehat f(\omega) = m(\omega) \wh \psi(\omega)$ a.e. $\omega \in \R^d.$
\item[b)] If $m$ is a measurable $\Lambda^\perp$-periodic function in $\R^d$ such that $m\wh \psi \in L^2(\R^d)$, the function $f$ defined by $\widehat f(\omega) = m(\omega) \wh \psi(\omega)$ a.e. $\omega \in \R^d$ satisfies $f\in S_\Lambda(\psi).$ 
\end{itemize}
\end{lem}

This result can be found in Theorem 2.14 in \cite{dDR1994} for the case $\Lambda = \Z^d$. Their proof can be modified to give the Lemma.

\begin{example} \label{example1}
	Let $\F =\{f_1, f_2 \}$ where $\wh f_1 = \chi_{[0, 1/2)}$ and $\wh f_2 = \chi_{[1/2, 1)}$. Take $\varphi\in L^2(\R)$ given by $\wh \varphi = \chi_{[0,1)}$. Let
	$$
	m_1(\omega)= \sum_{k\in \Z} \chi_{[0, 1/2)}(\omega + k) \quad \text{and} \quad 	m_2(\omega)= \sum_{k\in \Z} \chi_{[1/2, 1)}(\omega + k)\,.
	$$
	The functions $m_j, j=1,2$ are measurable and $\Z$-periodic. Moreover, $\wh f_1 = m_1 \wh \varphi$ and $\wh f_2 = m_2 \wh \varphi$. By $b)$ of Lemma \ref{characterization} both $f_1$ and $f_2$ belong to $S_\Z (\varphi)$. Hence, $\E(\F,S_\Z (\varphi)) = 0$ and also $\E^*(\F, \Z, 1) = 0.$ By \eqref{Inequality2}, $\E^*(\F, \Z/2, 1) = 0$ and \eqref{Inequality2} is an equality in this case.
\end{example}
    	
\begin{example} \label{example2}
	Let $\F =\{f_1, f_2 \}$ where $\wh f_1 = \chi_{[0, 1/2)}$ and $\wh f_2 = \chi_{[1, 3/2)}$. Take $\varphi\in L^2(\R)$ given by $\wh \varphi = \chi_{[0,2)}$. Let
	$$
	m_1(\omega)= \sum_{k\in \Z} \chi_{[0, 1/2)}(\omega + 2k) \quad \text{and} \quad 	m_2(\omega)= \sum_{k\in \Z} \chi_{[1, 3/2)}(\omega +2k)\,.
	$$
The functions $m_j, j=1,2$ are measurable and $2\Z$-periodic. Moreover, $\wh f_1 = m_1 \wh \varphi$ and $\wh f_2 = m_2 \wh \varphi$. By $b)$ of Lemma \ref{characterization} both $f_1$ and $f_2$ belong to the $S_{\Z/2} (\varphi)$. As in the previous example, $\E^*(\F, \Z/2,1)=0.$
	
	On the other hand $\E^*(\F, \Z, 1) > 0$. To prove this assume that  $\E^*(\F, \Z, 1) = 0$. By Theorem 2.1 in \cite{ACHM2007} there exists $\varphi \in L^2(\R)$ such that $\E(F,S_\Z(\varphi)) = \E^*(\F, \Z, 1)=0$. For this to hold we must have that $f_1$ and $f_2$ must belong to $S_\Z(\varphi)$. By part $a)$ of Lemma \ref{characterization}, there exist $\Z$-periodic functions $\eta_1$ and $\eta_2$ such that for $\omega \in \R$,
	\begin{equation} \label{Eq18}
		\chi_{[0,1/2)}(\omega)= \wh {f_1}(\omega)=\eta_1(\omega) \wh \varphi(\omega) \quad \text{and} \quad \chi_{[1,3/2)}(\omega)= \wh {f_2}(\omega)=\eta_2(\omega) \wh \varphi(\omega)\,.
	\end{equation}
    For $\omega \in [0,1/2)$, the first  equation in \eqref{Eq18} gives $\wh \varphi(\omega) \neq 0$ for all $\omega \in [0,1/2).$ The second equation in \eqref{Eq18} produces $\eta_2(\omega)=0$ for  all $\omega \in [0,1/2).$ Since $\eta_2$ is $\Z$-periodic we conclude that $\eta_2(\omega)=0$ for $\omega \in [1,3/2).$ This contradicts the second equation in \eqref{Eq18}.
\end{example}

\begin{example} \label{example3}
	We will show that for each $h\in \R$  irrational there exists a set of data $\F_h\subset L^2(\R)$ such that
	$$
	0 = \E^*(\F_h, h\Z, 1) < \E^* (\F_h, \Z, 1)\,.
	$$
	Let $f= \chi_{[-1/2, 1/2]}$ and $\F_h =\{f, t_h f \}$. It is clear that $f$ and $t_h f$ belong to $S_{h\Z}(f)$, the $h\Z$-invariant subspace generated by $f$. Therefore $\E^*(\F_h, h\Z, 1)=0.$
	
	We want to show that $\E^*(\F_h, \Z, 1) > 0$. If  $\E^*(\F_h, \Z, 1) = 0$, by Theorem 2.1 in \cite{ACHM2007} there exists $\varphi \in L^2(\R)$ such that $\E(F_h,S_\Z(\varphi))= \E^*(\F_h, \Z, 1)= 0$. For this to hold we must have that $f$ and $t_h f$ must belong to $S_\Z(\varphi)$.  By part $a)$ of Lemma \ref{characterization}, there exist $\Z$-periodic functions $\eta_1$ and $\eta_2$ such that for a.e. $\omega \in \R$,
	$$
		 \wh {f}(\omega)=\eta_1(\omega) \wh \varphi(\omega) \quad \text{and} \quad  \wh {t_h f}(\omega)=\eta_2(\omega) \wh \varphi(\omega)\,.
	$$
	Since $\wh f (\omega) = \text{Sinc} (\omega)$ and $\wh{t_hf}(\omega) = e^{-2\pi i h \omega} \text{Sinc} (\omega)$, we have
	$$
	\text{Sinc} (\omega)=\eta_1(\omega) \wh \varphi(\omega) \quad \text{and} \quad  e^{-2\pi i h \omega} \text{Sinc}=\eta_2(\omega) \wh \varphi(\omega)\,.
	$$
	Let $\Omega$ be the set where $\text{Sinc} (\omega) \neq 0$ which satisfies $|\R \setminus \Omega|=0.$ On this set $\Omega$
	$$
	e^{-2\pi i h \omega} = \frac{\eta_2(\omega)}{\eta_1(\omega)}\,.
	$$
	This is impossible because the function in the right hand side of the above equation is $\Z$-periodic while the function in the left is not.
	
	Observe that the same proof works for any $f\in L^2(\R)$ such that the support of $\wh f$ is $\R$.
\end{example}

\subsection{Two dimensional examples with an isometry}

Now we will look at two examples in which we allow for different lattices - not necessarily a refinement of one another.
The examples show that lattices play a significant role in the computation of the error. 

In order to present them, we need to recall the notion of {\em range function} which was introduced by Helson and Lowdenslager in \cite{HL58, Hel1964} and developed by Bownik in \cite{Bow2000}.
Let $\Lambda$ be a full rank lattice in $\R^d$ and let $\Omega_{\Lambda} \subset \wh{\R^d}$ be a Borel section of $\wh{\R^d} / \Lambda^\perp$. Recall the definition of the isometric isomorphism $\T_{\Lambda}$ given in equation \eqref{ii}. 
A range function is a map
$$
\J : \Omega_{\Lambda} \to \{\textnormal{closed subspaces of } \ell_2(\Lambda^\perp)\} .
$$
A range function $\J$ is said to be measurable if the family $\pj_{\J(\omega)} \in \B(\ell_2(\Lambda^\perp))$ of orthogonal projections onto $\J(\omega)$ is measurable.

A closed subspace $V$ of $L^2(\R^d)$ is $\Lambda$-invariant if and only if there exists a measurable range function $\J$ such that
$$
V = \{ f\in L^2(\R^d): \T_{\Lambda} f(\omega) \in \J(\omega) \  \text{a.e} \  \omega \in \Omega_\Lambda \}.
$$
Moreover, if $V=S_\Lambda(\mathcal A)$ for some countable set $\mathcal A$ of $L^2(\R^d)$ the measurable range function associated to $V$ is given by
$$
\J(\omega) = \overline {\text{span} }\{ \T_\Lambda \varphi (\omega) : \varphi \in \mathcal A \}\, \quad \text{a.e} \ \omega \in \Omega_\Lambda\,.
$$
Good references for this result are Proposition 1.5 in \cite{Bow2000} and Theorem 3.10 in \cite{CP2010}.


Let $\ell(S_\Lambda(\mathcal A))$ denote the length of $S_\Lambda(\mathcal A)$, that is the smallest  number of generators of $S_\Lambda(\mathcal A)$. It that can be seen that  $$S_\Lambda(\mathcal{A}) = \esssup_{w \in \Omega_{\Lambda}} \text{dim}(\J(\omega)).$$


\begin{example} \label{example4}
Let $\Lambda = \Z^2$ and $\Lambda_1 = R\Z^2$ be a rotation of $\Z^2$ by $\displaystyle\frac{\pi}{6}$. Let $\gamma := \left(\frac{\sqrt{3}}{2}, \frac{1}{2}\right)$ and let $\Omega = [0,1)\times [0,1)$, and  $\Omega_1= R\Omega$ be fundamental domains for $\Lambda$ and $\Lambda_1$ respectively. Observe that since $R$ is unitary, the dual lattice of $\Lambda_1$ is itself. Let $B_1, B_2  \subset \Omega$ be the following circles (see Figure \ref{fig1}):
\begin{align*}
B_1 &= B\left((0,0),\frac{1}{25}\right) + \left(\frac{2}{25}, \frac{3}{10}\right)\\
B_2 &= B_1 + \gamma = B\left((0,0),\frac{1}{25}\right) + \left(\frac{\sqrt3}{2}+\frac{2}{25}, \frac{1}{2}+\frac{3}{10}\right).
\end{align*}
This example is based on the fact that $B_2 = B_1 + \gamma$, but for the lattice $\Z^2$ both belong to the same fundamental domain while for $R\Z^2$ they don't. That is, $B_1 \in \Omega \cap R\Omega$, while $B_2 \in \Omega \setminus R\Omega$.

\begin{figure}
\begin{center}
	\includegraphics[width=.8\textwidth]{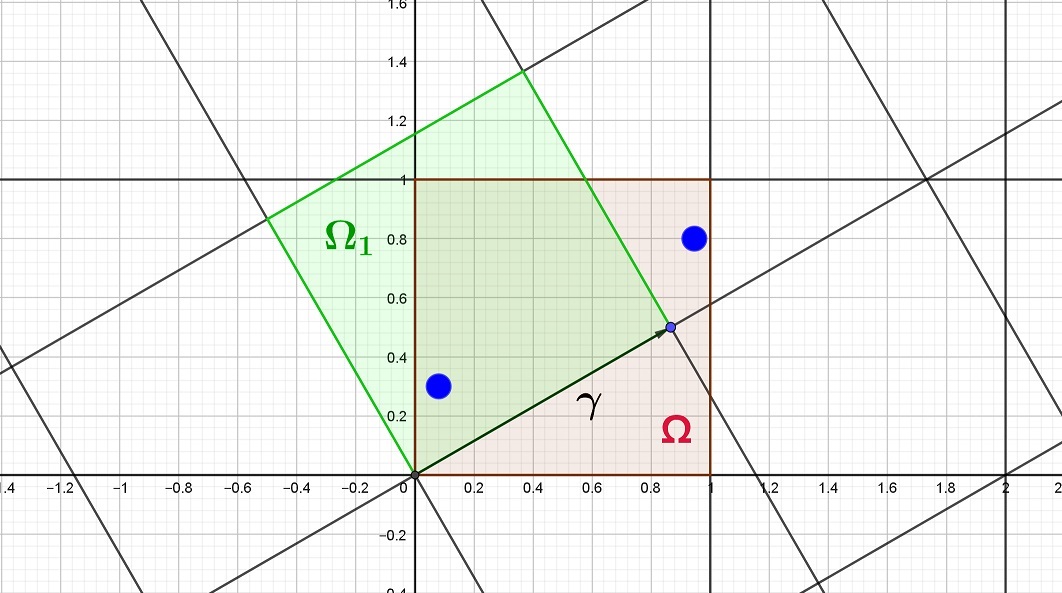}
\end{center}
\caption{The sets of Example 5.}\label{fig1}
\end{figure}

Let $\F = \{f_1, f_2\} \subset L^2(\R^2)$ be defined by $\wh{f_i} = \chi_{B_i}, \ i=1,2$. Hence, for the lattice $\Lambda=\ZZD$, 
$$ \T_\Lambda [f_1](\omega) = \left\{\begin{array}{ll} e_{(0,0)} & \omega \in B_1\\
0 &  \omega \in \Omega\setminus B_1 \end{array}\right. \qquad 
\T_\Lambda [f_2](\omega) = \left\{\begin{array}{ll} e_{(0,0)} & \omega \in B_2\\
0 & \omega \in \Omega \setminus B_2 \end{array}\right. . $$
This means that, for {\em a.e.} $\omega \in \Omega$, the dimension of $\J(\omega)$ is less or equal than 1, and the length of $S(\F)$ is $\ell(S(\F)) = 1$.

 If $\Phi = \{\varphi\}$ with 
$$\T_\Lambda [\varphi](\omega) = \left\{\begin{array}{ll} e_{(0,0)} & \omega \in B_1\cup B_2\\
0 & \omega \in \Omega \setminus (B_1\cup B_2)\end{array}\right. , $$
we have that $\|f_1 - \Proj_{S_\Lambda(\varphi)}(f_1)\| = \|f_2 - \Proj_{S_\Lambda(\varphi)}(f_2)\| = 0$ a.e. $\omega \in \Omega$ and hence $S(\varphi)$ approximates te data $\F$ perfectly.

However, if instead of the lattice $\Lambda=\ZZD$ we take the lattice $\Lambda_1 = R\ZZD$, then 
$$ \T_{\Lambda_1}[f_1](\omega) = \left\{\begin{array}{ll} e_{(0,0)} & \omega \in B_1\\
0 & \omega \in \Omega_1 \setminus B_1 \end{array}\right. \qquad 
\T_{\Lambda_1}[f_2](\omega) = \left\{\begin{array}{ll} e_{(1,0)} & \omega \in B_1\\
0 & \omega \in \Omega_1 \setminus B_1 \end{array}\right. $$
So, for $\omega \in B_1$, the dimension of $\J(\omega)$ is 2; hence tht length of $S_{\Lambda_1}(\F)$ is $2$ and therefore, if we want to approximate with a single function,  the error can never be $0$. Hence, $\E^*(\F, \Z^2,1)= \E^*(\F, R\Z^2,1).$
\end{example}

This example induces us to believe that to minimize the error of approximation we should search for the lattice $\Lambda$ that minimizes the length of $S_{\Lambda}(\F)$. As we show in the next example, this is not necessarily true.

\begin{example} \label{example5}


Let the lattices and fundamental domains be as in Example \ref{example4}, and let $\gamma_1 = \left(\frac{\sqrt{3}}{2}, \frac{1}{2}\right)$ and $\gamma_2 = 2\cdot\left(\frac{\sqrt{3}}{2}, \frac{1}{2}\right)$, 
Consider the following sets (see Figure \ref{fig2}):
\begin{align*}
B_1 &= B\left((0,0),\frac{1}{25}\right) + \left(\frac{2}{25}, \frac{3}{10}\right)\\
B_2 &= B_1+\gamma_1 \\
B_3 &=  B_1+\gamma_2 \\
Q_1 &= \text{square of side}\  \frac{1}{25} \text{with lower left vertex at}\ \left(\frac{1}{5},\frac{1}{25}\right)\\
Q_2 &= Q_1 + (1,0).
\end{align*}

\begin{figure}[h!]
\begin{center}
	\includegraphics[width=.8\textwidth]{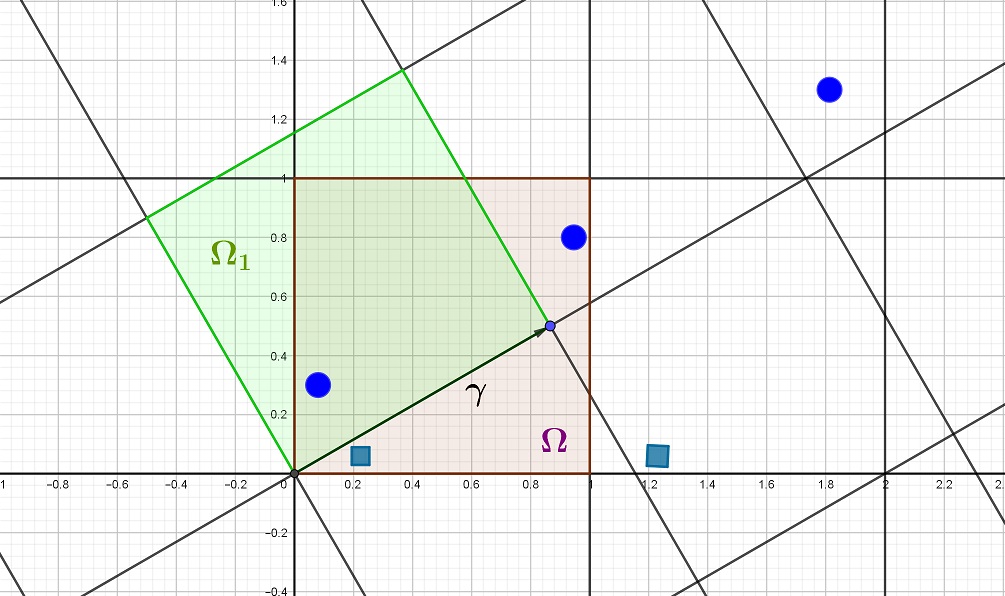}
\end{center}
\caption{The sets of Example 6.}\label{fig2}
\end{figure}

Let $\F = \{f_1, f_2, f_3, f_4, f_5\} \subset L^2(\R^2)$ be defined by 
$$
\begin{array}{c}
	\widehat{f_1} = \chi_{B_1} \, , \ 
	\widehat{f_2} = \chi_{B_1}+ \varepsilon\cdot \chi_{B_2} \, , \
	\widehat{f_3} = \chi_{B_1}+ \varepsilon\cdot \chi_{B_3} \, ,\\
	\widehat{f_4} = \chi_{Q_1} \, , \
	\widehat{f_5} = \chi_{Q_2} .
\end{array}
$$

Hence, for the lattice $\ZZD$, and $\Omega = [0,1)\times [0,1), $
\begin{align*}
\T_{\Z^2}[f_1](\omega) & = \left\{\begin{array}{ll} e_{(0,0)} & \omega \in B_1\\
0 & \omega \in \Omega \setminus B_1 \end{array}\right. \\
\T_{\Z^2}[f_2](\omega) & = \left\{\begin{array}{ll} e_{(0,0)} & \omega \in B_1\\
\varepsilon \cdot e_{(0,0)} & \omega \in B_2\\
0 & \omega \in \Omega \setminus (B_1\cup B_2)  \end{array}\right. \\
\T_{\Z^2}[f_3](\omega) & = \left\{\begin{array}{ll} e_{(0,0)} & \omega \in B_1\\
\varepsilon\cdot e_{(1,1)} & \omega \in B_3 - (1,1) \\
0 & \omega \in \Omega \setminus \big((B_3 - (1,1))\cup B_1 \big)\end{array}\right. \\
\T_{\Z^2}[f_4](\omega) & = \left\{\begin{array}{ll} e_{(0,0)} & \omega \in Q_1\\
0 & \omega \in \Omega \setminus Q_1  \end{array}\right. \\
\T_{\Z^2}[f_5](\omega) & = \left\{\begin{array}{ll} e_{(1,0)} & \omega \in Q_1\\
0 & \omega \in \Omega \setminus Q_1  \end{array}\right. 
\end{align*}

This means, that for almost every $\omega$ the dimension of $\J(\omega)$ is less or equal than 2, and  on $Q_1$ the dimension is $2$, so, the length of $S_{\Z^2}(\F)$ is 2.

Since $\T_{\Z^2}[f_4]$ and $\T_{\Z^2}[f_5]$ are orthogonal in $Q_1$, if we want to approximate with a single function, for $\omega \in Q_1$ the best we can do is to take the line right in the middle, so the best approximation error will be the area of $B_1$.

However, if instead of the lattice $\ZZD$ we take the lattice $\Lambda= R\ZZD$, where $R$ is a rotation of $\displaystyle \frac{\pi}{6}$, with fundamental domain $\Omega_1 = R \Omega$, 
\begin{align*}
\T_\Lambda [f_1](\omega) & = \left\{\begin{array}{ll} e_{(0,0)} & \omega \in B_1\\
0 & \omega \in \Omega_1 \setminus B_1  \end{array}\right. \\
\T_\Lambda [f_2](\omega) &= \left\{\begin{array}{ll} e_{(0,0)} + \varepsilon \cdot e_{(1,0)} & \omega \in B_1\\
0 & \omega \in \Omega_1 \setminus B_1  \end{array}\right. \\
\T_\Lambda [f_3](\omega) &= \left\{\begin{array}{ll} e_{(0,0)} + \varepsilon \cdot e_{(2,0)} & \omega \in B_1\\
0 & \omega \in \Omega_1 \setminus B_1  \end{array}\right. \\
\T_\Lambda [f_4](\omega) &= \left\{\begin{array}{ll} e_{(0,-1)} & \omega \in Q_1 + (-\frac12, \frac{\sqrt{3}}{2})\\
0 & \omega \in \Omega_1 \setminus (Q_1 + (-\frac12, \frac{\sqrt{3}}{2})) \end{array}\right. \\
\T_\Lambda[f_5](\omega) &= \left\{\begin{array}{ll} e_{(1,-1)} & \omega \in Q_1 - \gamma_1 + (-\frac12, \frac{\sqrt{3}}{2}) \\
0 & \omega \in \Omega_1 \setminus (Q_1 - \gamma_1 + (-\frac12, \frac{\sqrt{3}}{2}))  \end{array}\right.
\end{align*}

Hence, in $B_1$ we have three linearly independent vectors, and on every other point only one or none, so the length of $S_\Lambda(\F)$ is 3. But if we look at the vectors, and compute the error that occurs by approximating with a single line, the error will be $\varepsilon^2$ times de area of $B_1$.

So, the set $\F$ of this example can be better approximated in the second lattice, even though the length is bigger.
\end{example}

\newpage
\noindent
\textbf{Davide Barbieri}\\
Universidad Aut\'onoma de Madrid, 28049 Madrid, Spain\\
{\tt davide.barbieri@uam.es}

\vspace{1ex}
\noindent
\textbf{Carlos Cabrelli}\\
Universidad de Buenos Aires and 
IMAS-CONICET, Consejo  Nacional de Investigaciones Cient\'ificas y T\'ecnicas, 1428 Buenos Aires,  Argentina\\
{\tt carlos.cabrelli@gmail.com}

\vspace{1ex}
\noindent
\textbf{Eugenio Hern\'andez}\\
Universidad Aut\'onoma de Madrid, 28049 Madrid, Spain\\
{\tt eugenio.hernandez@uam.es}

\vspace{1ex}
\noindent
\textbf{Ursula Molter}\\
Universidad de Buenos Aires and 
IMAS-CONICET, Consejo  Nacional de Investigaciones Cient\'ificas y T\'ecnicas, 1428 Buenos Aires,  Argentina\\
{\tt umolter@dm.uba.ar}

\end{document}